\documentclass[a4paper,11pt,twoside]{paper}
\usepackage[T1]{fontenc} 
\newif\ifRR\RRfalse
\RRtrue
\ifRR
  \usepackage{RR}
\fi
\usepackage{hyperref}
\usepackage{amsmath,amssymb,amsthm,mathrsfs}
\newtheorem{prop}{Proposition}[section]
\newtheorem{coro}[prop]{Corollary}
\newtheorem{theo}[prop]{Theorem}
\newtheorem{defin}[prop]{Definition}
\newtheorem{lem}[prop]{Lemma}

\newtheoremstyle{etoile}{\parskip}{\parskip}{\itshape}
                       {0pt}{\bfseries\sffamily}{.}{ }{}
\theoremstyle{etoile}

\newcommand\bbox{\hfill\rule{2mm}{2mm}}

\makeatletter
\@addtoreset{equation}{section}
\makeatother

\newcommand\C{\mathbb{C}}
\newcommand\Hc{\mathcal{H}}
\newcommand\Pb{\mathbb{P}}
\newcommand\Sc{\mathcal{S}}
\newcommand\Zb{\mathbb{Z}}

\def\DD{\displaystyle}
\def\egaldef{\stackrel{\mbox{\tiny def}}{=}}
\begin{document}
\ifRR
\RRdate{Avril 2010}
\RRauthor{Guy Fayolle\thanks{INRIA Paris-Rocquencourt, Domaine de Voluceau, BP 105, 78153 Le Chesnay Cedex, France. {\tt Guy.Fayolle@inria.fr}}    \and
        Kilian Raschel \thanks{ Laboratoire de Probabilit\'es et
        Mod\`eles Al\'eatoires, Universit\'e Pierre et Marie Curie,
        4 Place Jussieu, 75252 Paris Cedex 05, France.
        \texttt{kilian.raschel@upmc.fr}}}

\RRtitle{Sur l'holonomie ou l'algébricité de fonctions génératrices comptant des marches aléatoires dans le quart de plan}
\RRetitle{On the Holonomy or Algebraicity of  Generating Functions Counting Lattice Walks in the Quarter-Plane}
\titlehead{On the Holonomy or Algebraicity of  Counting Generating Functions}
\RRresume{Dans deux articles récents \cite{BMM,BK}, il a été montré que les fonctions génératrices de comptage (CGF) pour 23 marches avec pas unité, évoluant dans le quart de plan et associées a un groupe fini de transformations bi-rationnelles, sont holonomes et même algébriques dans 4 cas -- notamment celui de la marche dite de Gessel. Or il s'avère que le type d'équations fonctionnelles vérifiées par ces CGF est apparu  il y a environ 40 ans, dans le contexte probabiliste des marches aléatoires. Une méthode de résolution fut proposée dans  \cite{FIM}, basée à la fois sur des techniques algébriques et sur la réduction à des problèmes aux limites. Récemment cette méthode a été développée dans un cadre combinatoire dans  \cite{Ra}, où une étude fouillée  des formes explicites des CGF a été menée. L'objet du présent article est de retrouver la nature des CGF bivariées par une application directe de théorèmes donnés dans  \cite{FIM}.}

\RRabstract{In two recent works \cite{BMM,BK}, it has been shown that the counting generating functions (CGF) for the 23 walks with small steps confined in a quarter-plane and associated with a finite group of birational transformations are holonomic, and even algebraic in 4 cases -- in particular for the so-called Gessel's walk. It turns out that the type of functional equations satisfied by these CGF appeared in a probabilistic context almost 40 years ago. Then a method of resolution was proposed in \cite{FIM}, involving at once algebraic tools and a reduction to boundary value problems. Recently this method has been developed in a combinatorics framework in \cite{Ra}, where a thorough study of the explicit expressions for the CGF is proposed. The aim of this paper is to derive the nature of the bivariate CGF by a direct use of some general theorems given in  \cite{FIM}.}

\RRmotcle{Algébrique, holonome, fonction génératrice, marche discrète homogène par morceaux, quart de plan, couverture universelle, fonction élliptique de Weierstrass.}

\RRkeyword{Algebraic, holonomic, generating function, piecewise homogeneous lattice walk, quarter-plane, universal covering, Weierstrass elliptic function.}

\RRprojet{Imara}

\RRtheme{\THNum}
\RRNo{7242}
\URRocq
\makeRR
\else
\pagestyle{empty}

\title{On the Holonomy or Algebraicity of  Generating Functions Counting Lattice Walks in the Quarter-Plane}
\author{Guy Fayolle\thanks{INRIA Paris-Rocquencourt, Domaine de Voluceau, BP 105, 78153 Le Chesnay Cedex, France. {\tt Guy.Fayolle@inria.fr}}    \and
        Kilian Raschel \thanks{ Laboratoire de Probabilit\'es et
        Mod\`eles Al\'eatoires, Universit\'e Pierre et Marie Curie,
        4 Place Jussieu, 75252 Paris Cedex 05, France.
        \texttt{kilian.raschel@upmc.fr}}}

\maketitle

\begin{abstract}
In two recent works \cite{BMM,BK}, it has been shown that the counting generating functions (CGF) for the 23 walks with small steps confined in a quarter-plane and associated with a finite group of birational transformations are holonomic, and even algebraic in 4 cases -- in particular for the so-called Gessel's walk. It turns out that the type of functional equations satisfied by these CGF appeared in a probabilistic context almost 40 years ago. Then a method of resolution was proposed in \cite{FIM}, involving at once algebraic technics and reduction to boundary value problems. Recently this method has been developed in a combinatorics framework in \cite{Ra}, where a thorough study of the explicit expressions for the CGF is proposed. The aim of this paper is to derive the nature of the bivariate CGF by a direct use of some general theorems given in  \cite{FIM}.
\end{abstract}


\emph{AMS $2000$ Subject Classification: primary 05A15; secondary 30F10, 30D05.}
\newpage
\fi
\section{Introduction} \label{Introduction}
The enumeration of planar lattice walks is a classical topic in combinatorics. For 
a given set $\Sc$ of allowed jumps (or steps), it is a matter of counting the number of paths starting from some point and ending at some arbitrary point in a given time, and possibly restricted to some regions of the plane. A first basic and natural question arises: how many such paths exist? A second question concerns the nature of the associated counting generating functions (CGF): are they rational, algebraic, holonomic (or D-finite, i.e.\ solution of a linear differential equation with polynomial coefficients)?

For instance, if no restriction is made on the paths, it is well-known that the CGF are rational and easy to obtain explicitly. As an other example, if the walks are supposed to remain in a half-plane, then the CGF can also be computed and turn out to be algebraic in this case, see e.g.~\cite{MBM2}.

Next it is quite natural to consider walks evolving in a domain formed by the intersection of two half-planes, for instance the positive quarter-plane $\Zb_{+}^{2}$. In this case, the problems become more intricate and multifarious results appeared: indeed, some 
walks admit of an algebraic generating function, see e.g.\ \cite{FL1} and \cite{Gessel} 
for the walk with step set $\Sc=\{(-1,0),(1,1),(0,-1)\}$ and starting from $(0,0)$, 
whereas some others admit a CGF which is even not holonomic, see e.g.\ \cite{MBM2} for the walk with $\Sc=\{(-1,2),(2,-1)\}$ and starting from the point $(1,1)$. 

In this framework, M.\ Bousquet-M\'{e}lou and M.\ Mishna have initiated in \cite{BMM} a systematic study of the nature of the walks confined to $\Zb_{+}^{2}$, starting from the origin and having steps of size $1$, which means that $\Sc$ is included in the set $\{(i,j) : |i|,|j|\leq 1\}\setminus \{(0,0)\}$. Examples of such walks are shown in figures~\ref{Fig_hol} and~\ref{Fig_alg} below.

A priori, there are $2^{8}$ such models. In fact, after eliminating trivial  cases and models equivalent to walks confined to a half-plane, and noting also that some models are obtained from others by symmetry, it is shown  in \cite{BMM} that one is left with  $79$ 
inherently different problems to analyze.

A common starting point to study these $79$ walks relies on the following analytic approach. Let $q(i,j,k)$ denote the number of paths in $\Zb_{+}^{2}$ starting from $(0,0)$ 
and ending at $(i,j)$ at time $k$. Then the corresponding CGF 
     \begin{equation}  \label{CGF}
          F(x,y,z)=\sum_{i,j,k\geq 0}f(i,j,k)x^{i}y^{j}z^{k}
     \end{equation}
satisfies the functional equation obtained in~\cite{BMM}
     \begin{equation}  \label{FE}
          K(x,y)F(x,y,z) =  c(x)F(x,0,z)+\widetilde{c}(y)F(0,y,z) + c_0(x,y),
     \end{equation}
where 
\begin{equation*}
\begin{cases}
\DD K(x,y)  =  xy\Bigg[\sum_{(i,j)\in\Sc} x^iy^j-1/z\Bigg] , \\[4ex]
\DD c(x)  = \sum_{(i,-1)\in\Sc} x^{i+1}  ,\\[3ex]
\DD \widetilde{c}(y)=\sum_{(-1,j) \in\Sc} y^{j+1}, \\[0.2cm]
\DD c_0(x,y) = -\delta F(0,0,z)-x y/z ,
\end{cases} 
\end{equation*}
with $\delta=1$ if $(-1,-1)\in\Sc$ [i.e.\ a south-west jump exists], $\delta=0$ otherwise.

For $z=1/|\Sc|$, (\ref{FE}) plainly belongs to the generic class of functional 
equations (arising in the probabilistic context of random walks) studied and solved 
in the book~\cite{FIM}, see section \ref{recall-A} of our appendix.
For general values of $z$, the analysis of (\ref{FE}) for the 79 above-mentioned walks has been  carried out in \cite{Ra}, where the integrand of the integral representations is studied in detail, via a complete  characterization of ad hoc conformal gluing functions.

It turns out that one of the basic tools to solve equation~(\ref{FE}) is to consider the group exhaustively studied in~\cite{FIM}, and originally proposed in \cite{MAL} where it was called the \emph{group of the random walk}. Starting with this approach, the authors  of \cite{BMM} consider the group of birational transformations leaving invariant the generating function $\sum_{(i,j)\in\Sc}
x^{i}y^{j}$, which is precisely the group $W=\langle \xi,\eta\rangle$ generated by
     \begin{equation*} 
          \xi(x,y)= \Bigg(x,\frac{1}{y}\frac{\sum_{(i,-1)\in\Sc}x^{i}}
          {\sum_{(i,+1)\in\mathcal{S}}x^{i}}\Bigg),\qquad
          \eta(x,y)=\Bigg(\frac{1}{x}\frac{\sum_{(-1,j)\in\mathcal{S}}y^{j}}
          {\sum_{(+1,j)\in\mathcal{S}}y^{j}},y\Bigg).
     \end{equation*}
Clearly $\xi\circ\xi=\eta\circ\eta=\text{id}$, and $W$ is a dihedral group of even order 
 larger than 4. In \cite{BMM} this order is calculated for each of the above-mentioned 79 
cases: 56 walks admit an infinite group, while the group of the 23 remaining ones is finite.

It is also proved in~\cite{BMM} that among these 23 walks,  $W$ has order 
4 for 16 walks (the ones with a step set having a vertical symmetry), $W$ has order 
6 for 5 walks (the 2 at the left in figure~\ref{Fig_hol} and the 3 ones at the left in 
figure~\ref{Fig_alg}), and $W$ has order 8 for the 2 walks on the right in
figures~\ref{Fig_hol} and~\ref{Fig_alg}. Moreover, for these 23 walks, the answers to both main questions (explicit expression \emph{and} nature of the CGF~(\ref{CGF})) are 
known. In particular, the following results exist.

\begin{theo}[\cite{BMM}]\label{22_BMM}
For the 16 walks with a group of order 4 and for the 3 walks in figure~\ref{Fig_hol}, the 
formal trivariate series~(\ref{CGF}) is holonomic non-algebraic. For the 
3 walks on the left in figure~\ref{Fig_alg}, the trivariate series~(\ref{CGF}) is algebraic. 
\end{theo}
The proof of this theorem relies on skillful algebraic manipulations together with the calculation of adequate \emph{orbit} and \emph{half-orbit} sums.

\begin{theo}[\cite{BK}]\label{1_BK}
For the so-called \emph{Gessel's walk} on the right in figure~\ref{Fig_alg}, 
the formal trivariate series~(\ref{CGF}) is algebraic. 
\end{theo}
The proof given in \cite{BK} involves a powerful computer algebra system (Magma), allowing to carry out dense calculations.

The main goal of our  paper will be to present another proof of theorems~\ref{22_BMM} and \ref{1_BK} for the bivariate generating function $(x,y)\mapsto F(x,y,z)$, by means of the general and powerful approach proposed in the book~\cite{FIM}. In particular,  we are going to show how chapter 4 of this book (which deals with walks having a finite group) yields rather directly the nature of the bivariate generating functions
(\ref{CGF}) for the 23 walks associated with a finite group. 

The rest of the paper has a simple organization. Indeed, section~\ref{Recall} recalls
the results of~\cite{FIM}, used in Section~\ref{Nature}, where we find the nature of the counting generating functions coming in~(\ref{FE}).

\begin{figure}[!ht]
\begin{center}
\begin{picture}(340.00,65.00)
\includegraphics{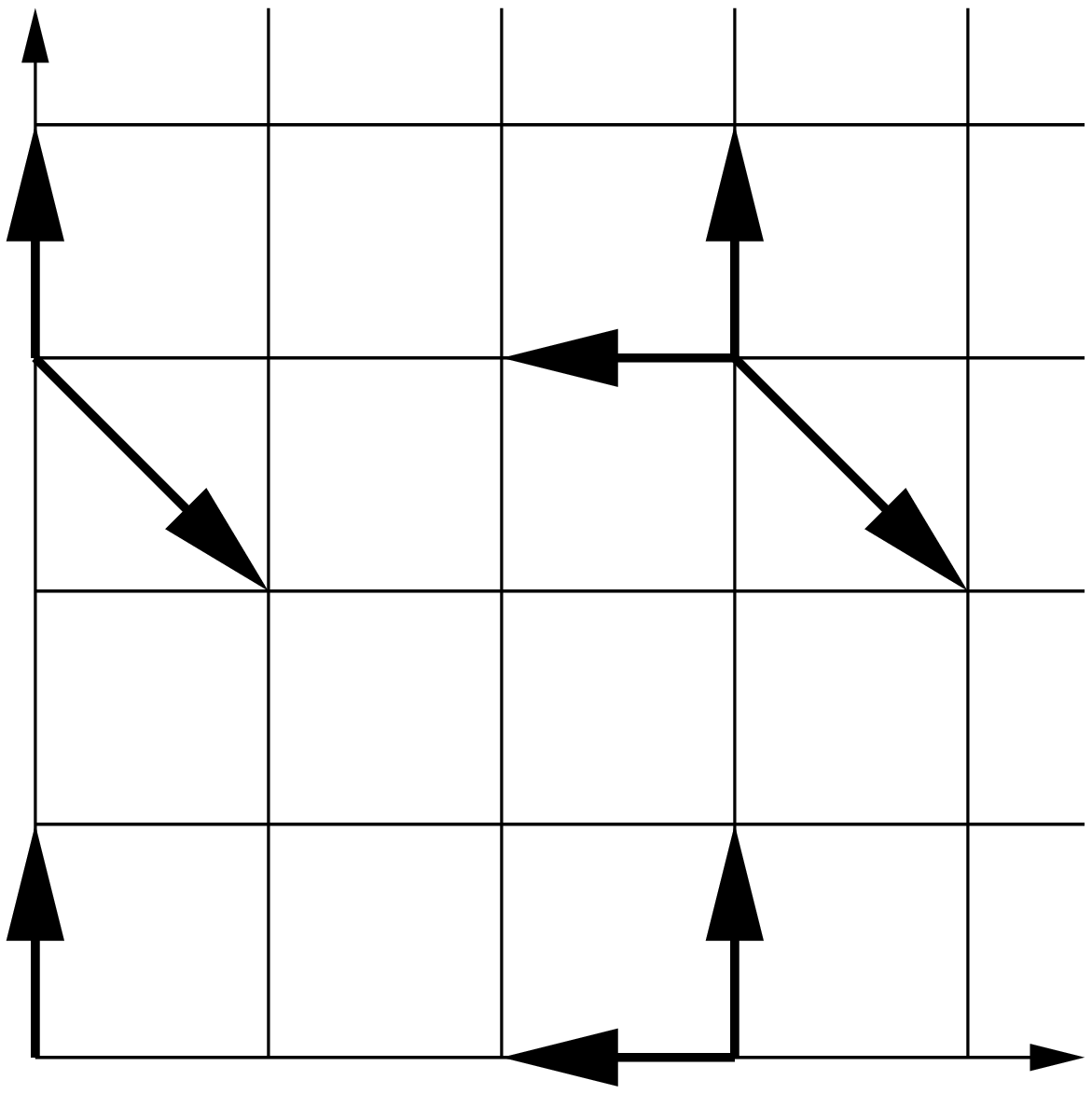}
\hspace{45mm}
\includegraphics{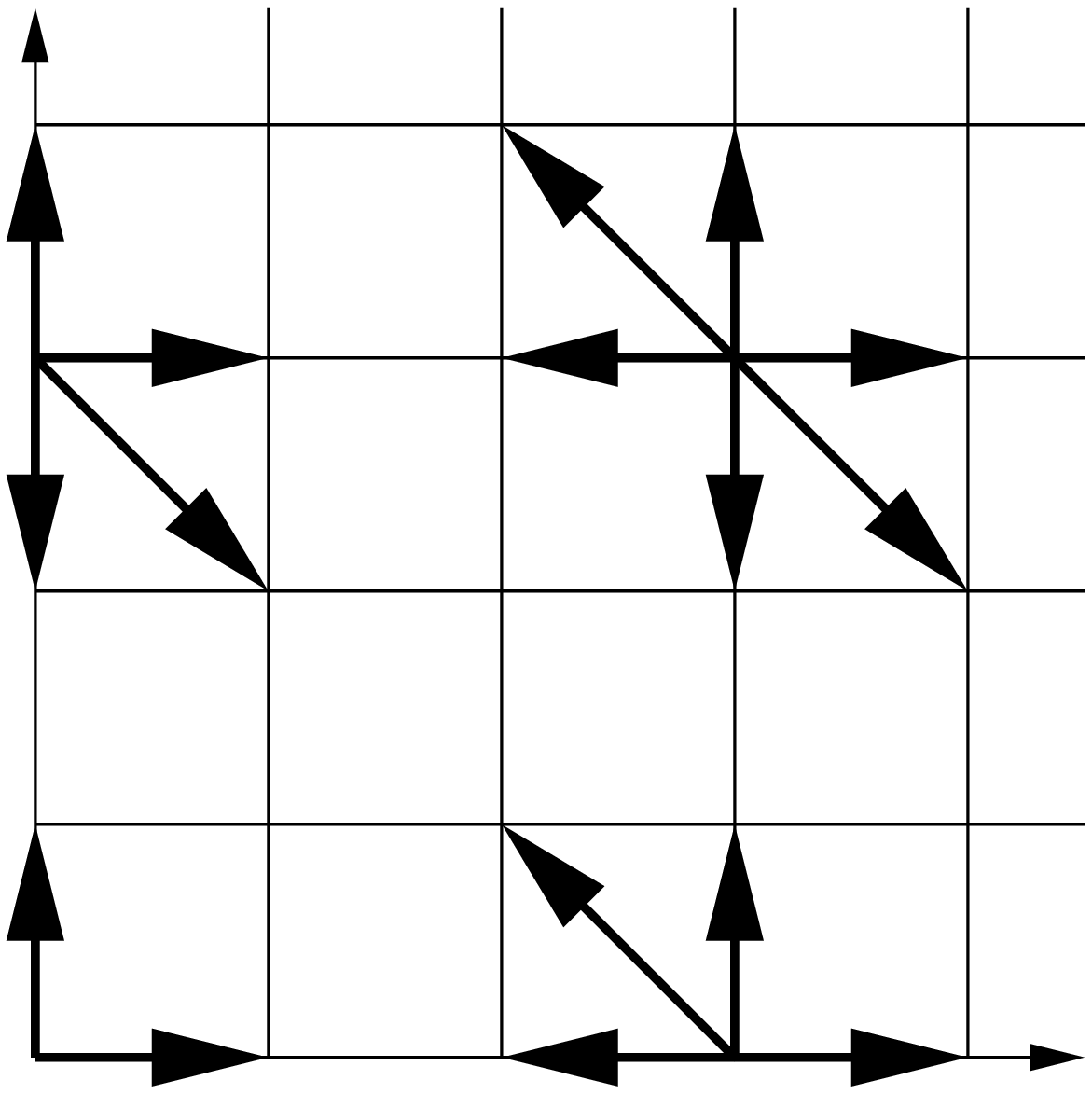}
\hspace{45mm}
\includegraphics{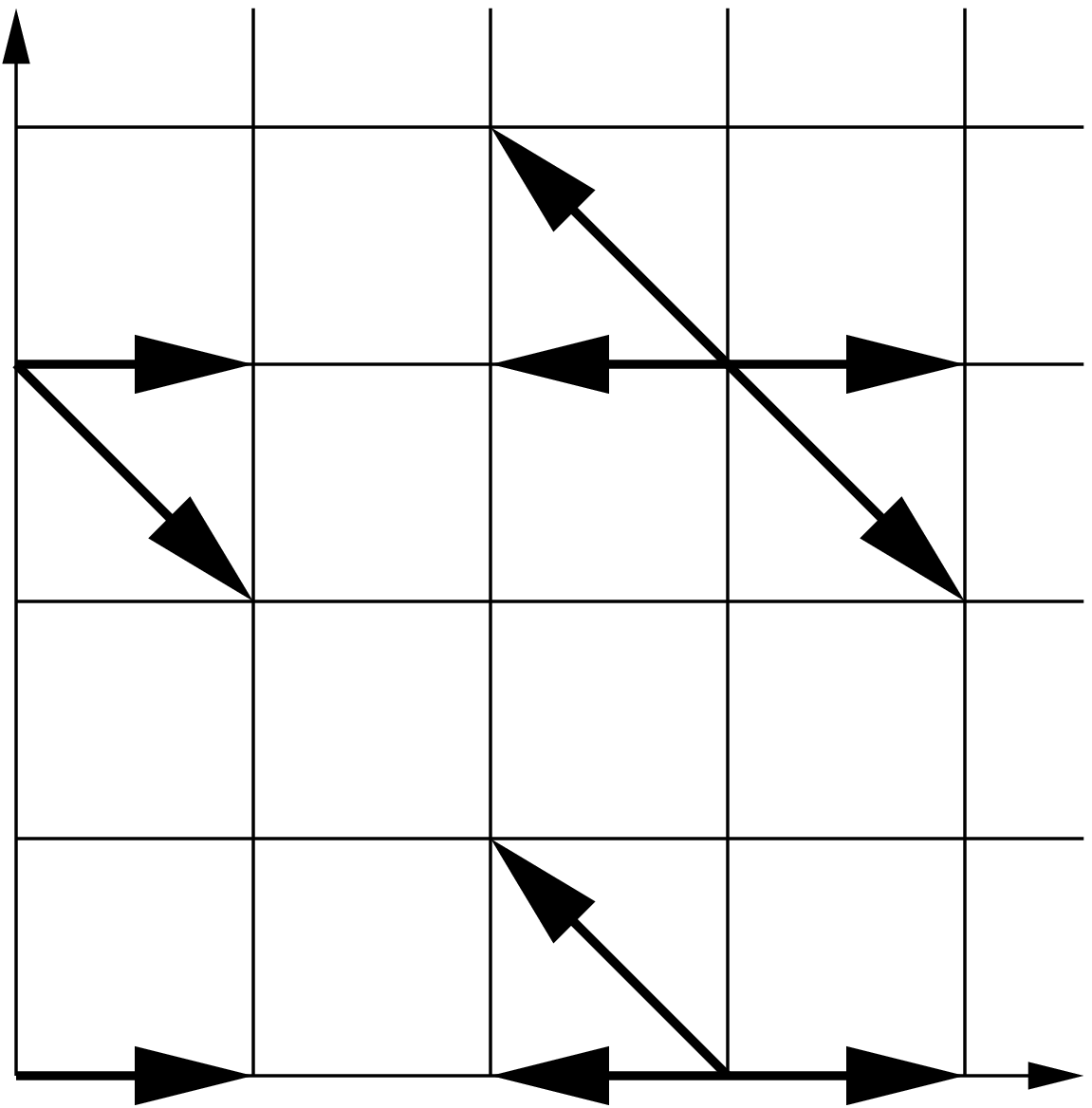}
\end{picture}
\end{center}
\vspace{-3mm}
\caption{On the left,  2 walks with a group of order 6.
         On the right, 1 walk with a group of order 8}
\label{Fig_hol}
\end{figure}

\begin{figure}[!ht]
\begin{center}
\begin{picture}(390.00,70.00)
\includegraphics{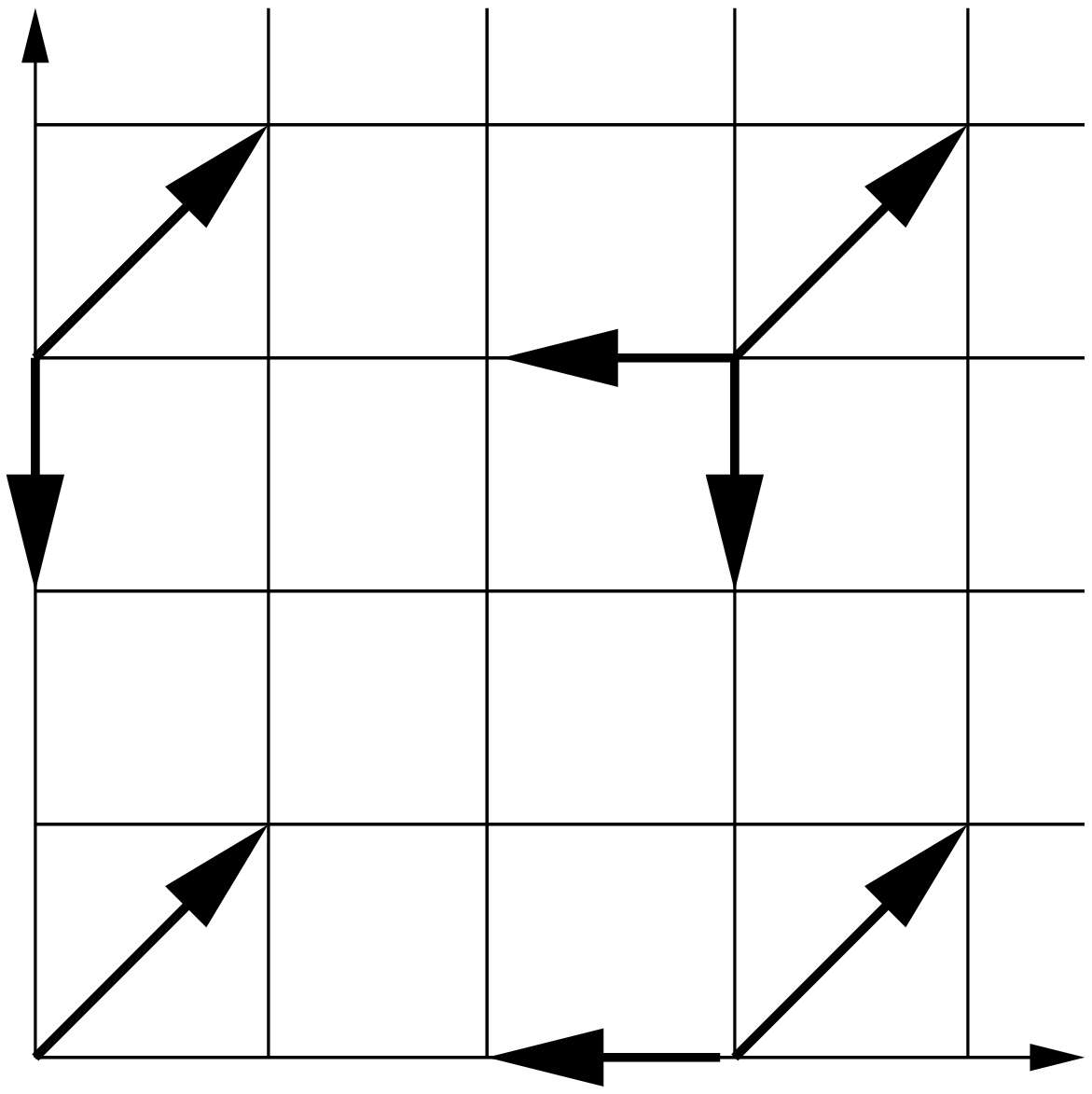}
\hspace{35mm}
\includegraphics{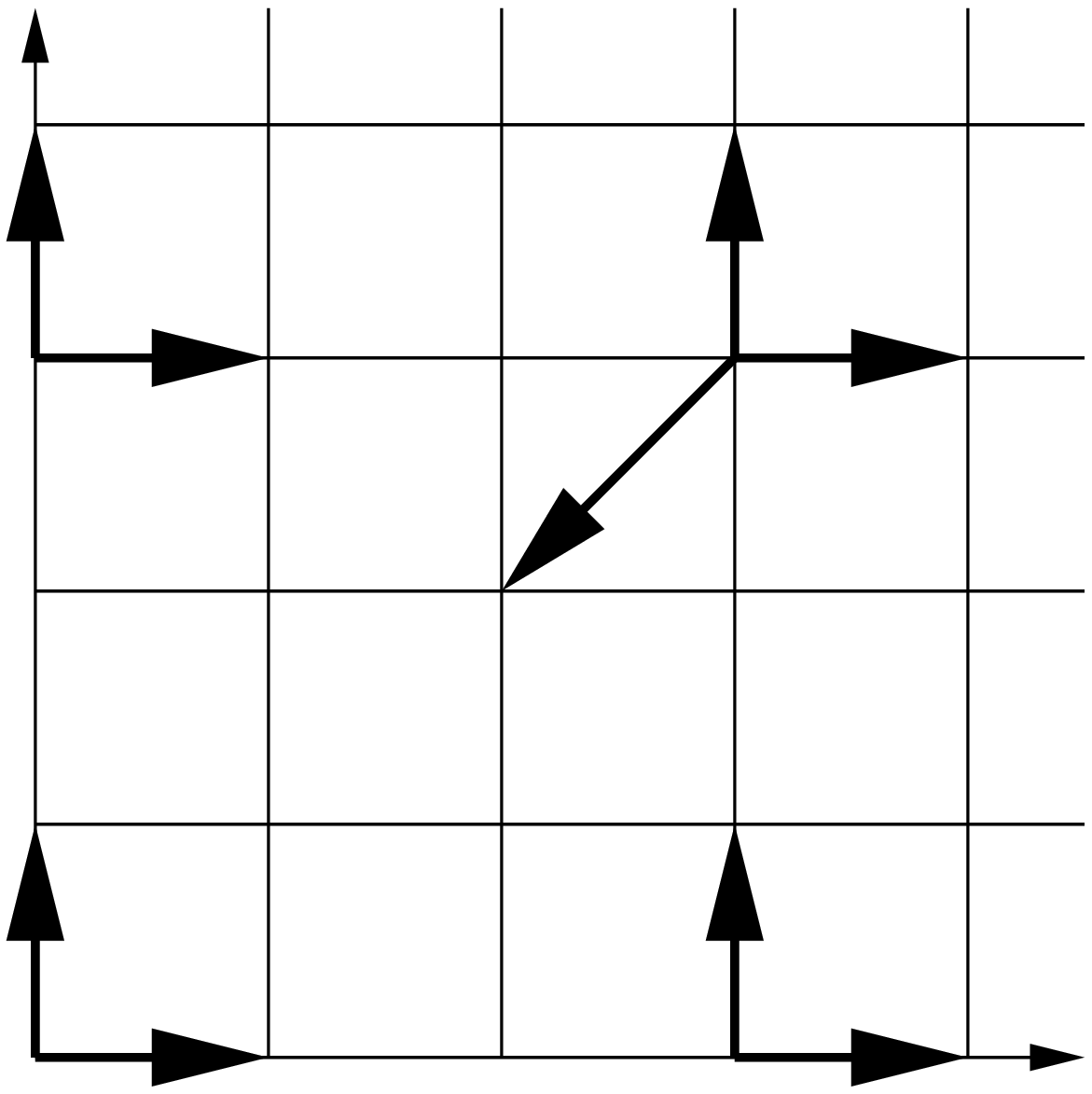}
\hspace{35mm}
\includegraphics{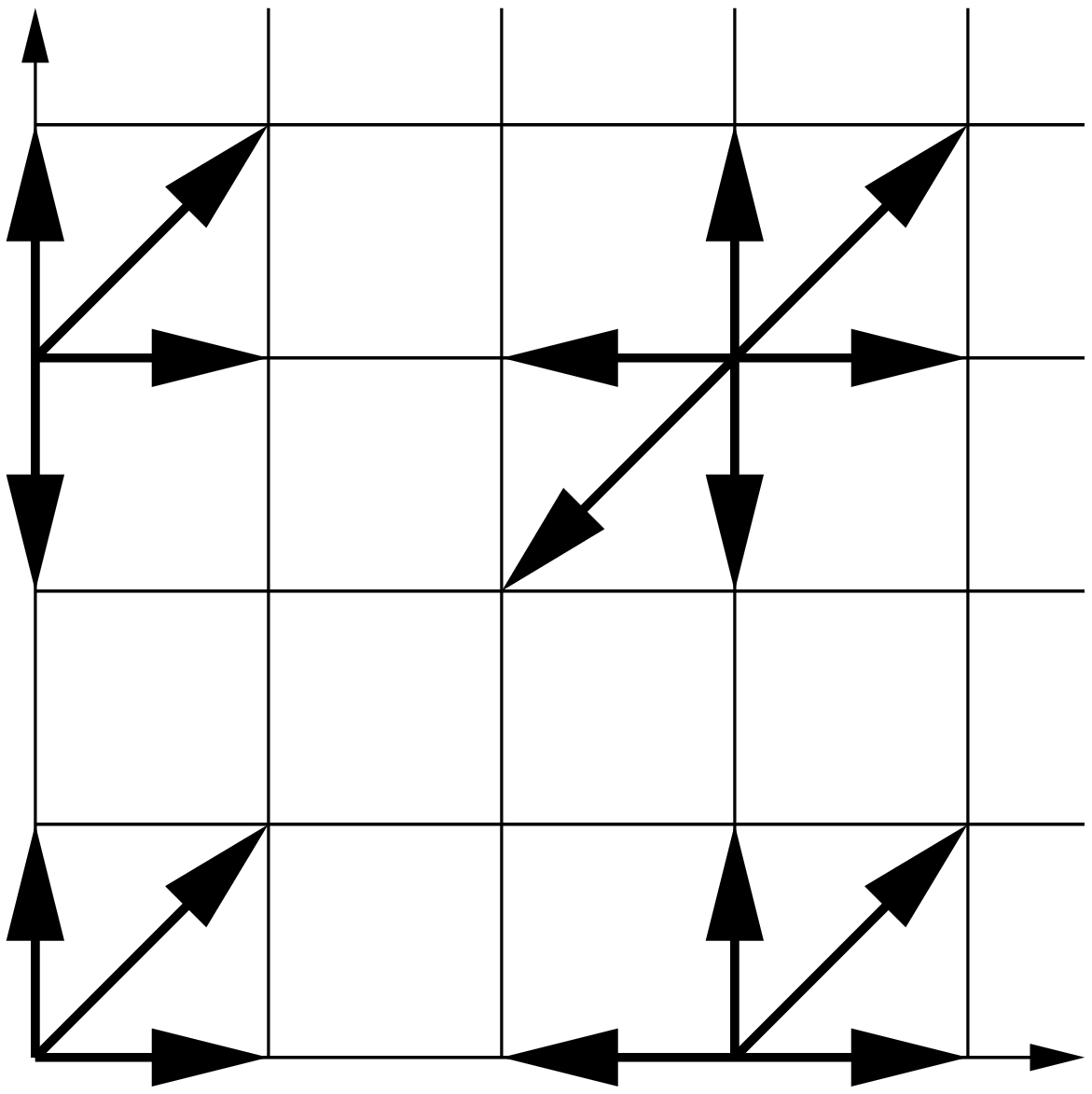}
\hspace{35mm}
\includegraphics{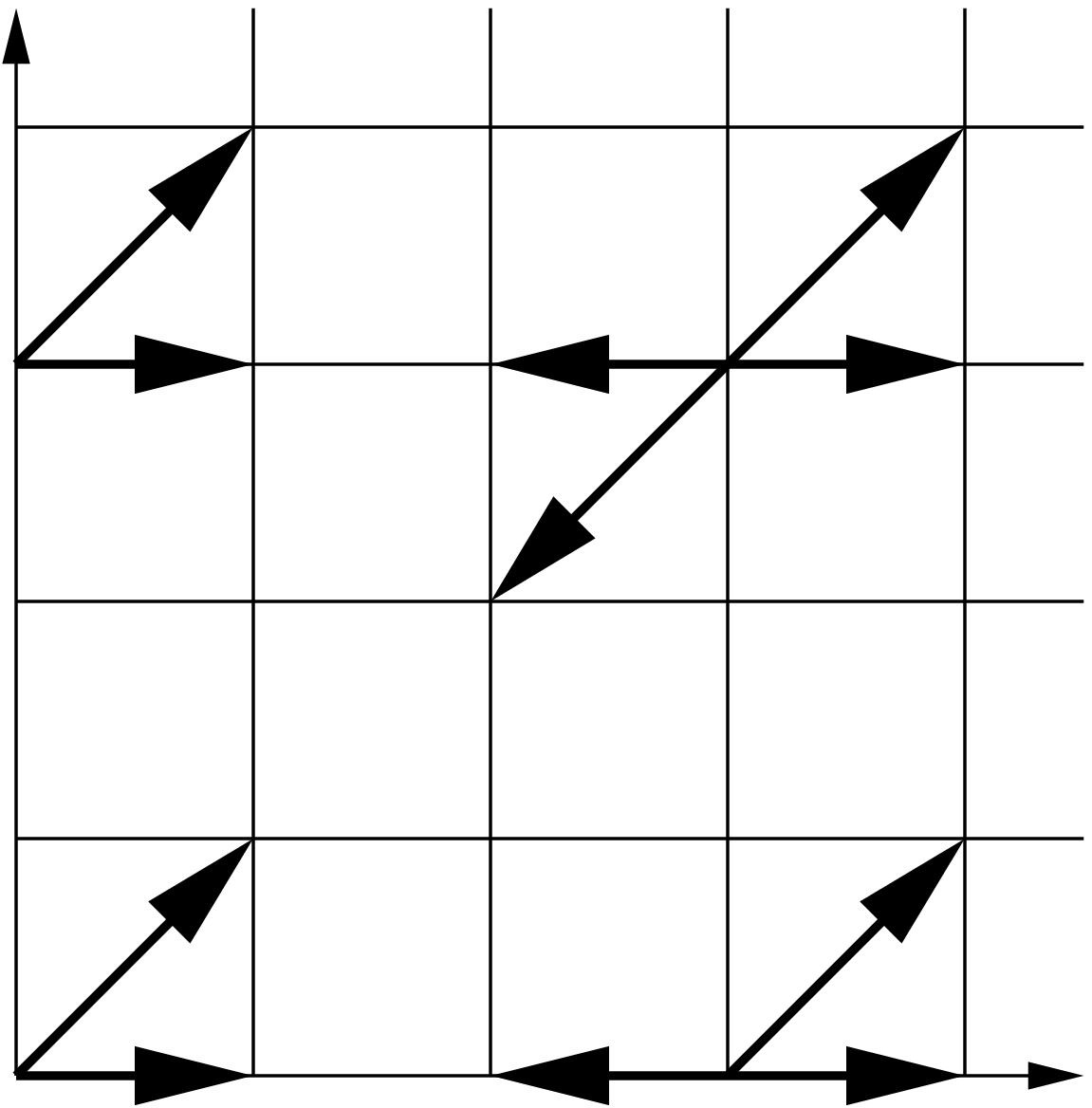}
\end{picture}
\end{center}
\vspace{-3mm}
\caption{On the left,  3 walks with a group of order 6. 
         On the right, 1 walk with a group of order 8}
\label{Fig_alg}
\end{figure}


\section{Some results of \cite{FIM} with miscellaneous extensions} \label{Recall}
In this section we use the notation of appendix \ref{recall-A}, where is presented  the generic functional equation (\ref{A1-FE}) studied in~\cite{FIM}. It clearly contains~(\ref{CGF}) as a particular case, since the coefficients $q$, $\widetilde{q}$ and $q_{0}$  occuring in~(\ref{A1-FE}) may depend on the \emph{two} variables $x,y$. Moreover, although this be not crucial, the unknown function $\pi(x)$ [resp.\ $\widetilde{\pi}(y)$] is not necessarily equal to $\pi(x,0)$ [resp.\ $\widetilde{\pi}(0,y)$]. 

Chapter~5 of~\cite{FIM} explains how to solve (\ref{A1-FE})  by reduction to  boundary value problems of Riemann-Hilbert-Carleman type, and explicit expressions for the unknown functions are obtained.

Independently, \emph{when the group of the walk is finite},  chapter~4 of~\cite{FIM} provides another approach allowing to characterize completely the solutions of (\ref{A1-FE}), and also to give \emph{necessary and sufficient conditions} for these solutions \emph{to be rational or algebraic}. We shall proceed along these lines in the sequel. 
 \subsection{The group and the genus}\label{GROUP} Let $\C\,(x,y)$ be the field of rational functions in $(x,y)$ over $\C$. Since $Q$ is assumed to be irreducible in the general case, the quotient field $\C\,(x,y)$ denoted by 
 $\C_Q(x,y)$ is also a field.
 \begin{defin}
The \emph{group of the random walk} is the Galois group  $\Hc=\langle \xi,\eta\rangle$ of automorphisms of   $\C_Q(x,y)$ generated by  $\xi$ and $\eta$ given by
     \begin{equation*}
          \xi(x,y)=\Bigg(x,\frac{1}{y}\frac{\sum_{i} p_{i,-1}x^{i}}
          {\sum_{i} p_{i,+1}x^{i}}\Bigg),\ \ \ \ \
          \eta(x,y)=\Bigg(\frac{1}{x}\frac{\sum_{j} p_{-1,j}y^{j}}
          {\sum_{j} p_{+1,j}y^{j}},y\Bigg) .
     \end{equation*}
\end{defin}
 Let
\begin{equation*}
\delta = \eta \xi .
\end{equation*}
Then $\Hc$ has a normal cyclic subgroup $\Hc_0=\{\delta^i, i\in\Zb\,\}$, which is finite or infinite, and $\Hc / \Hc_0$ is a group of order 2. 
When the group $\Hc$ is finite of order $2n$, we have $\delta^{n}=\text{id}$.
We shall write $f_\alpha=\alpha(f)$ for any automorphism $\alpha\in\Hc$ and any function 
$f\in\C_Q(x,y)$.

In \cite{BMM}, the group $W$ is introduced only in terms of two birational transformations. The difference between the two approaches is not only of a formal character. In fact $W$ is defined in all of $\C^2$, whereas $\Hc$ in  \cite{FIM} would act only on an algebraic curve of the type  
\[
\bigg\{(x,y)\in \C^2 : R(x,y) = xy \bigg[\sum_{i,j} r_{i,j} x^i y^j -1/z\bigg] =0\bigg\}.
\]
An immediate question arises: do the notions of order and finiteness coincide in the respective approaches? In general, the answer is clearly no. Indeed (adapting \cite{FIM}, lemma 4.1.1) for the group to be of order $4$ it is necessary and sufficient to have
\begin{equation*}
\left| \begin{array}{lll}
r_{1,1} & r_{1,0} & r_{1,-1} \\ r_{0,1} & r_{0,0}-1/z & r_{0,-1} \\ r_{-1,1}
& r_{-1,0} & r_{-1,-1} \end{array} \right| = 0. 
\end{equation*}
and this condition depends on $z$, while $W$ is independent of $z$. 
On the other hand, if $W$ is finite so is $\Hc$, and conversely if $\Hc$ is infinite so is $W$.
In addition, with an obvious notation,
\begin{equation}\label{HW}
\textit{Order}(\Hc)\le \textit{Order}(W).
\end{equation}
In the sequel, we shall encounter groups of order not larger than $8$.

\medskip
Besides, in \cite{FIM}, the algebraic curve is associated with a Riemann surface which  is of \emph{genus} $g=0$ (the sphere) or $g=1$ (the torus). Accordingly, the \emph{universal covering} of this surface is 
\begin{itemize}
\item the Riemann sphere $\Pb^1$ if $g=0$; 
\item the finite complex plane $\C_{\omega}$ if $g=1$.
\end{itemize}
All 23 random walks we  consider in this paper correspond indeed to $g=1$, and an efficient uniformization is briefly described in appendix \ref{recall-B}.  Moreover, they have a finite group, say of order $2n$ (this will be made precise in section \ref{Nature}).

\subsection{Algebraicity and holonomy}

For any $h \in\C_Q(x,y)$, let the \emph{norm} $N(h)$ be defined as
\begin{equation}\label{eq-norm}
N(h)  \egaldef  \prod_{i=0}^{n-1} h_{\delta^i}. 
 \end{equation}

Written on $Q(x,y)=0$, equation (\ref{A1-FE}) yields the system
\begin{equation} \label{eq.pi}
 \begin{cases}
 \pi= \pi_\xi, \\[0.1cm]
 \pi_\delta - f \pi = \psi,
\end{cases}
 \end{equation}
where 
\begin{equation*}
f= \frac{q\widetilde{q}_\eta}{\widetilde{q}q_\eta} , \quad
 \psi = \frac{q_0\widetilde{q}_\eta}{q_\eta \widetilde{q}} - \frac{(q_0)_\eta}{q_\eta}.
\end{equation*}
It turns out that the analysis of~(\ref{A1-FE}) given in \cite{FIM} deeply differs, according 
$N(f)=1$ or $N(f)\ne1$. In section \ref{Nature}, as for the above-mentioned walks, it will be shown that we are in the case  $N(f)=1$ and we need the following important theorem stated in a terse form.
\begin{theo} \label{main_theorem_FIM}\quad If $N(f) = 1$, then the general
solution of the fundamental system (\ref{eq.pi})
has the form
\begin{equation*}
\pi = w_1 + w_2 + \dfrac{w}{r} , 
\end{equation*} where 
 
 -- the function  $w_1$ and $r$ are rational;

 -- the function $w_2$ is given by
\begin{equation} \label{def_w_2}
w_2 = \dfrac{\widetilde{\Phi}}{n}\,
\sum_{k=0}^{n-1}\left(\frac{\psi_{\delta^k}}{\prod_{i=1}^k
f_{\delta^i}}\right); 
\end{equation} 

-- on the universal covering $\C_{\omega}$, we have 
\begin{equation}
\label{def_Phi} \widetilde{\Phi}\left(\omega
+\frac{\omega_2}{2}\right) \, \egaldef \, 
\frac{\omega_1}{2 i \pi} \zeta(\omega ; \omega_1, \omega_3) - \frac{\omega}
{i \pi} \zeta \left( \frac{\omega_1}{2} ; \omega_1, \omega_3\right),
\end{equation} 
where $\zeta(\omega; \omega_1, \omega_3) \egaldef \zeta_{1,3}(\omega)$ stands for the classical Weierstrass $\zeta$-function (see \cite{HUR}) with quasi periods $\omega_1, \omega_3$;

-- the function $w$ is algebraic and satisfies the automorphy conditions
$ w = w_\xi = w_\delta$.
\bbox
 \end{theo}
The quantities $\omega_1, \omega_2, \omega_3$ have explicit forms (see~\cite{FIM}, lemmas 3.3.2 and 3.3.3) and the automorphism $\delta$ writes simply as the translation
 \[
 \delta(\omega) = \omega +\omega_3, \quad \forall \omega \in\C_\omega.
 \]
In addition, when the group is finite of order $2n$, there exists an integer $k$ relatively prime with $n$ satisfying the relation
\[
n\omega_3 = k\omega_2,
\]
which in turn allows to prove (see \cite{FIM}, lemma 4.3.5) that $\widetilde{\Phi}(\omega)$ defined in (\ref{def_Phi}) is \emph{not} algebraic in $x(\omega)$.  We are now in a position to derive the following corollary.
     
\begin{coro} \label{Coro1}
Assume the group is finite of order $2n$ and $N(f)=1$.  Then the solution $\pi(x)$ 
of~(\ref{A1-FE}) is holonomic. Moreover it is algebraic if and only if on $Q(x,y)=0$,
 \begin{equation}\label{CNS}
\sum_{k=0}^{n-1} \frac{\psi_{\delta^k}}{\prod_{i=1}^kf_{\delta^i}} =0. 
 \end{equation}
\end{coro}
In fact, the proof is an immediate consequence of the next lemma.
\begin{lem} \label{Holonomic}
In the uniformization given in appendix \ref{recall-B}, let us write
\[
\wp(\omega) \egaldef g(x(\omega)),
\]
where g is a fractional linear transform. Then 
 the function $w_{2}(\wp^{-1}(g(x)))$ defined by~(\ref{def_w_2}) is holonomic in $x$, where 
$\wp^{-1}$ denotes the elliptic integral inverse function of $\wp$.
\end{lem}

\begin{proof}
  Since $w_{2}$ is the product of $\widetilde{\Phi}$ by an algebraic function, it suffices to prove that $\widetilde{\Phi}(\wp^{-1}(g))$ is holonomic.  It is known that the class of holonomic functions is closed under indefinite integration (see e.g.~\cite{FLAJ}), thus it is enough to prove that $[\widetilde{\Phi}(\wp^{-1}(g))]'$ is holonomic. In fact, we are going to show that $[\widetilde{\Phi}(\wp^{-1}(g))]'$ is algebraic
  -- we recall that any algebraic function is holonomic. Indeed, use on the one hand that the derivative of $-\zeta_{1,3}$ is $\wp_{1,3}$, the Weierstrass elliptic function with periods $\omega_{1},\omega_{3}$, and on the other hand that 
 \[
    [\wp^{-1}]'(u)=1/[4u^3-g_{2}u-g_{3}]^{1/2},
\]
 $g_2$ and $g_3$ being the so-called invariants of $\wp$ --  for these two properties see e.g.~\cite{HUR}. In this way we get 
     \begin{equation*}
          -2i \pi [\widetilde{\Phi}(\wp^{-1}(g))]'=\frac{g'}{
          [4g^{3}-g_{2} g-g_{3}]^{1/2}}\big[\omega_{1}\wp_{1,3}(
          \wp^{-1}(g)-\omega_{2}/2)+2\zeta_{1,3}(\omega_{1}/2)\big].
     \end{equation*}
The function $g$ being rational, $g'/[4g^{3}-g_{2} g-g_{3}]^{1/2}$ is algebraic and
in order to conclude it is enough to prove that $\wp_{1,3}(\wp^{-1}(g)-\omega_{2}
/2)$ is algebraic. Since $n\omega_{3}=k\omega_{2}$, it is shown in \cite{FIM}, lemma~4.3.3 that both $\wp$ and $\wp_{1,3}$ are rational functions of the Weierstrass elliptic function with periods $\omega_{1},k\omega_{2}$. Then it is immediate that $\wp_{1,3}$ is also  algebraic function of $\wp$. Moreover, by the well-known addition theorem for the Weiertrass elliptic functions (see e.g.~\cite{HUR}), $\wp$ is a rational function of 
$\wp(\,\cdot\,+\omega_{2}/2)$, so that $\wp_{1,3}$ is also an algebraic function 
of $\wp(\,\cdot\,+\omega_{2}/2)$.  In particular $\wp_{1,3}(\wp^{-1}(g)-\omega_{2}/2)$ is algebraic and 
lemma~\ref{Holonomic} is proved.
\end{proof}

Of course similar results can be written for $\widetilde{\pi}$. In particular, it is easy to check that if  $N(f)=1$ then also $\widetilde{N}(f)=1$, where $\widetilde{N}(f)=1$ is defined in terms of $\widetilde{\delta}\egaldef\xi\eta$. Thus we are in a position to state the following general result.
\begin{theo} \label{Bivariate}
Assume the group is finite of order $2n$ and $N(f)=1$. Then the bivariate series 
$\pi(x,y)$ solution of~(\ref{A1-FE}) is holonomic. 
Furthermore $\pi(x,y)$ is algebraic if and only if~(\ref{CNS}) and
($\widetilde{\ref{CNS}}$) hold on $Q(x,y)=0$,  ($\widetilde{\ref{CNS}}$) denoting
 the condition for $\widetilde{\pi}(y)$ to be algebraic, obtained by symmetry from~(\ref{CNS}).
\bbox
\end{theo}
\section{Nature of the counting generating functions} \label{Nature}
We return now to equation (\ref{FE}) and make use of the machinery of section \ref{Recall} with 
\begin{equation}\label{eq-fpsi}
f= \frac{c\widetilde{c}_\eta}{\widetilde{c}c_\eta} , \quad
 \psi = \frac{c_0\widetilde{c}_\eta}{c_\eta \widetilde{c}} - \frac{(c_0)_\eta}{c_\eta}.
\end{equation}

Beforehand it is important to note that the variable $z$ plays here the role of a 
\emph{parameter}. Moreover, the CGF (\ref{CGF}) is well defined and analytic  in a region containing the domain 
\[
\{(x,y,z)\in \mathbb{C}^{3} : |x|\le1, |y|\le1, |z|<1/|\Sc|\}.
\]
The analysis in \cite{FIM} aimed at finding the stationay measure of the random walk, in which case  $z=1$. But it is possible to check \emph{mutatis mutandis} that most of the topological results remains valid at least for  $\Re(z)\ge0$ [see for instance \cite{FIMI} or \cite{Ra} for  time dependent problems].

\begin{lem}\label{NORM}
For  any finite $n$ and $f$ given by (\ref{eq-fpsi}), we have $N(f)=1$.
\end{lem}
\begin{proof}
Since $c$ [resp.\ $\widetilde{c}$\,] is a function solely of $x$ [resp.\ $y$], we have 
$c=c_\xi$ and $\widetilde{c}= \widetilde{c}_\eta$. Hence
\[f= \frac{c}{c_\eta} = \frac{c_\delta}{c_{\delta\eta}}=\frac{c_\delta}{c},
\] 
which yields immediately $N(f)=1$ from the definition (\ref{eq-norm}).
\end{proof}

In order to obtain theorems~\ref{22_BMM} and~\ref{1_BK} for the bivariate
function $(x,y)\mapsto F(x,y,z)$, it is thus enough, by theorem \ref{Bivariate}
and lemma \ref{NORM}, to prove the following final proposition.

\begin{prop}
For the 4 walks in figure \ref{Fig_alg}, (\ref{CNS}) and ($\widetilde{\ref{CNS}}$) hold
in $\C^{2}$. As for the 19 other walks (3 in figure \ref{Fig_hol}
and the 16  ones with a vertically symmetric step set $\mathcal{S}$),
(\ref{CNS}) or ($\widetilde{\ref{CNS}}$) do not hold on $\big\{(x,y)\in \C^{2}:
K(x,y)=xy\big[\sum_{(i,j)\in\Sc}x^{i}y^{j}-1/z\big]=0\big\}$, for any $z$.
\end{prop}

\begin{proof} The proof proceeds in three stages.
\paragraph{(i)} \emph{The walks in figure \ref{Fig_alg}}.

Let us check (\ref{CNS}) for the 4 walks in figure \ref{Fig_alg}.
We begin with the popular Gessel's walk, i.e.\ the rightest one in figure \ref{Fig_alg}. Here $c=1$ in~(\ref{FE}), so that by~(\ref{eq-fpsi}) 
$f=1$ and $\psi=[(xy)_{\eta }-xy]/z$. Moreover,  the order of the group $W$ was shown in \cite{BMM} to be equal to 8. According to (\ref{HW}), we know that $\textit{Order}(\Hc)\le 8$. On the other hand, by lemmas 4.1.1 and 4.1.2 of \cite{FIM}, a simple algebra (details are omitted) can show that the order of $\Hc$ cannot be equal to 4 neither to 6, which entails $\textit{Order}(\Hc)=8$, hence $n=4$.

Then
     \begin{equation}
     \label{c=1}
          \sum_{k=0}^{n-1} \frac{\psi_{\delta^k}}{\prod_{i=1}^{k}f_{\delta^i}}=
          \sum_{k=0}^{3}\psi_{\delta^{k}}=\frac{1}{z}\sum_{k=0}^{3}\big[(xy)_{
          \eta\delta^{k}}-(xy)_{\delta^{k}}\big]=-\frac{1}{z}\sum_{\rho\in \Hc}
          \text{sign}(\rho)(xy)_{\rho}.
     \end{equation}

The last quantity in (\ref{c=1}) can be understood as $(-1/z)$ times the orbit sum 
of $xy$ through the group $\Hc$, and it was remarked in \cite{BMM} that this 
orbit sum equals zero for Gessel's walk, which is equivalent to~(\ref{CNS}) in $\mathbb{C}^{2}$.

For the second walk in figure~\ref{Fig_alg}, $\textit{Order}(W)=\textit{Order}(\Hc)= 6$  ($n=3$) and still $c=1$, so that the above argument applies.

For the walk on the left in figure~\ref{Fig_alg}, we have
 \[
 \xi=(x,1/(xy)), \quad \eta=(1/(xy),y),
 \]
hence $\xi\eta=(1/(xy),x)$. Moreover $c=x$, thus  by using
(\ref{eq-fpsi}) we get $f=x^{2}y$ and $\psi=[y-(xy)^{2}]/z$. 
Next one easily computes $f_{\delta}=1/(xy^{2})$, $f_{\delta^{2}}=y/x$,
$\psi_{\delta}=[x-1/y^{2}]/z$ and $\psi_{\delta^{2}}=[1/(xy)-1/x^{2}]/z$. 
Finally an immediate calculation shows that~(\ref{CNS}) holds in $\C^{2}$.

As for the last walk in figure \ref{Fig_alg}, we could check $n=3$ and again
 (\ref{CNS}) holds in~$\C^{2}$.

Concerning equation  ($\widetilde{\ref{CNS}}$), only Gessel's 
walk needs some special care, since the 3 others give rise to symmetric conditions.
 In fact, from a direct calculation along the same lines as above, it is not difficult to see that 
($\widetilde{\ref{CNS}}$) holds in  $\C^{2}$.
\paragraph{(ii)} \emph{The 16 walks with a vertical symmetry.} 
\begin{figure}[!ht]
\begin{center}
\begin{picture}(340.00,65.00)
\includegraphics{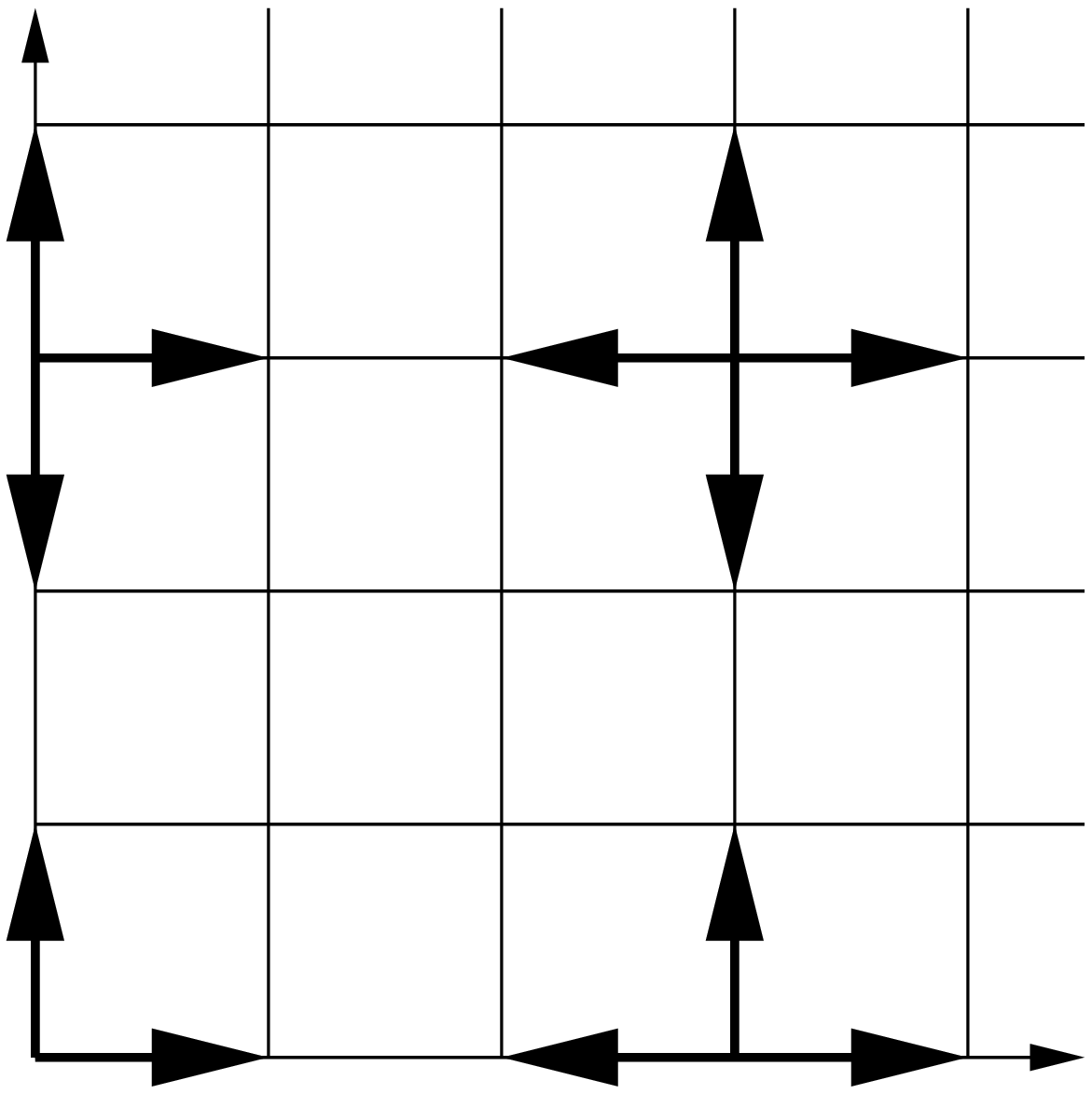}
\hspace{45mm}
\includegraphics{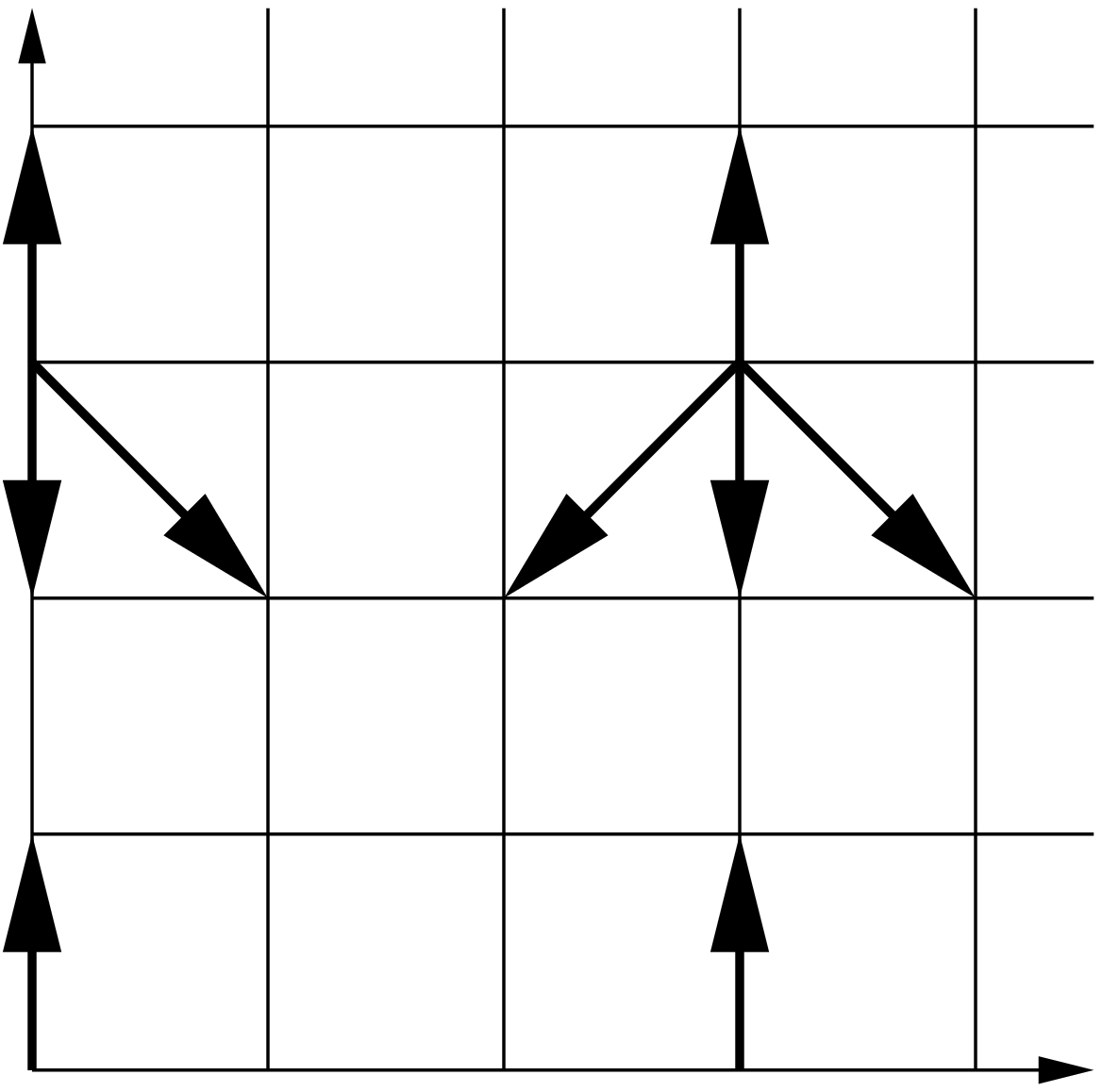}
\hspace{45mm}
\includegraphics{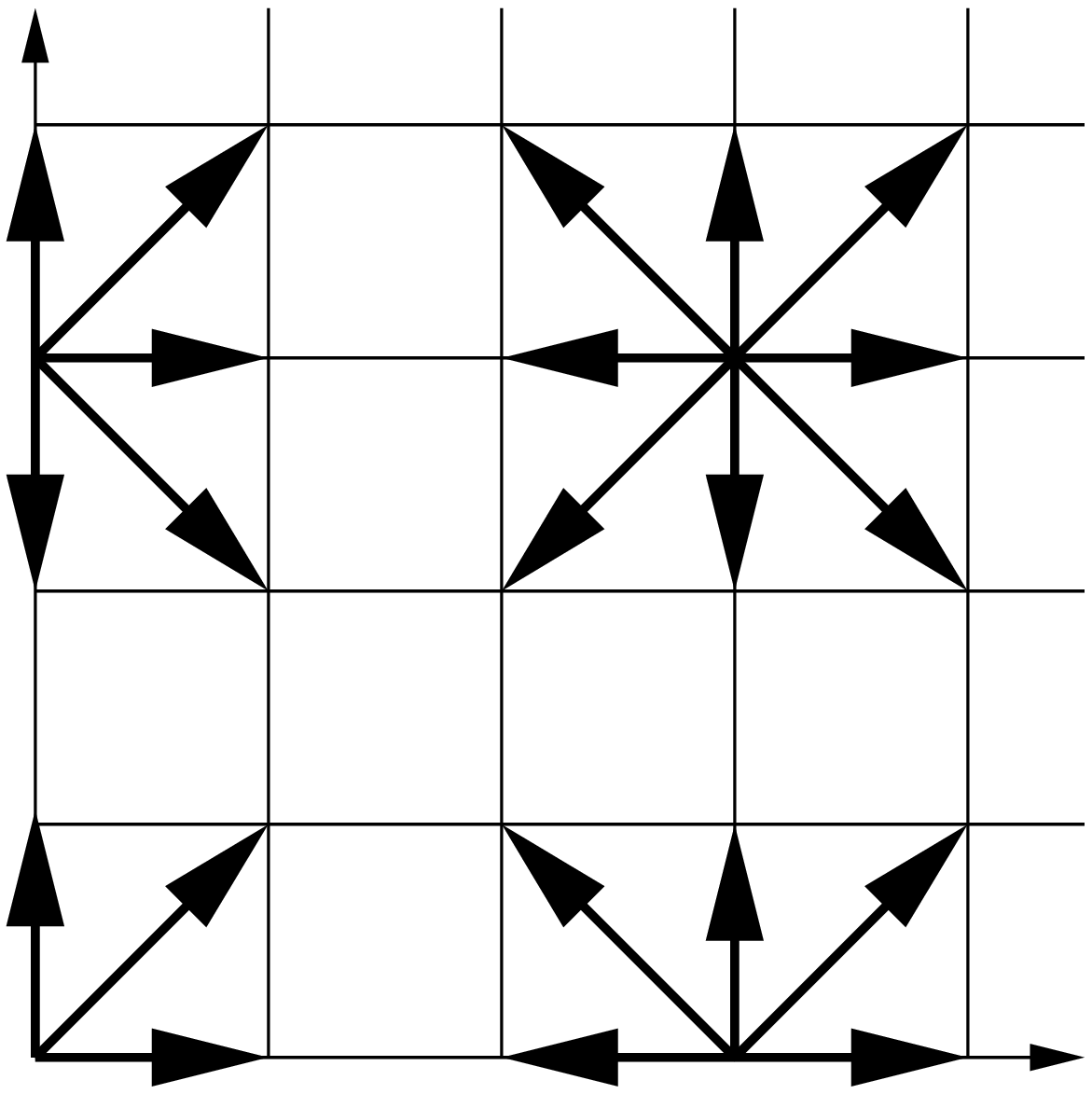}
\end{picture}
\end{center}
\vspace{-3mm}
\caption{Three examples of walks having a group of order 4, with symmetrical \emph{West} and \emph{East} jumps}
\label{Fig_four}
\end{figure} 

We will show the equality
     \begin{equation}
     \label{group_four_orbit}
          \sum_{k=0}^{n-1} \frac{\psi_{\delta^k}}{\prod_{i=1}^kf_{\delta^i}} =-\frac{x^{2}}{z c}
          (x-x_{\eta})(y-y_{\xi}),
     \end{equation} 
which clearly implies that (\ref{CNS}) will  not be satisfied on 
$\left\{(x,y)\in \C^{2}:K(x,y)=0\right\}$, for any $z$.

For these 16 walks we have $\textit{Order}(W)=4=\textit{Order}(\Hc)$, that is $n=2$.
Moreover, the vertical symmetry of the jump set $\mathcal{S}$ has two consequences~: 
first, necessarily  $\eta=(1/x,y)$; second, the coefficient $c$ in~(\ref{FE}) must be equal to $x$, $1+x^{2}$ or $1+x+x^{2}$. In addition, for any of these three possible values of $c$, we have $f=c/c_{\eta}=x^{2}$. Hence finally
     \begin{equation*}
          \sum_{k=0}^{n-1} \frac{\psi_{\delta^k}}{\prod_{i=1}^kf_{\delta^i}} =
          \psi+\frac{\psi_{\delta}}{f_{\delta}}=\frac{(xy)_{\eta}-xy}{zc_{\eta}}+
          \frac{(xy)_{\xi}-(xy)_{\delta}}{zc_{\xi}f_{\delta}}=\frac{y(x_{\eta}-x)}
          {zc_{\eta}}+\frac{y_{\xi}(x-x_{\eta})}{z c (x^{2})_{\delta}}.
     \end{equation*}
Factorizing by $x^{2}/c$ and using again $c/c_{\eta}=x^{2}$, we obtain
(\ref{group_four_orbit}).
\paragraph{(iii)}\emph{The 3  walks in figure \ref{Fig_hol}.} 

We are going to show that (\ref{CNS})
does not hold on $K(x,y)=0$, for any $z$.

By similar calculations as in (\ref{group_four_orbit}), we obtain that for the $2$ first walks in
figure \ref{Fig_hol} (which are such that $n=3$), 
     \begin{equation*}
          \sum_{k=0}^{n-1} \frac{\psi_{\delta^k}}{\prod_{i=1}^kf_{\delta^i}} =
          \frac{(x-y^{2})(1-xy)(y-x^{2})}{z y^{3}t},
     \end{equation*}
with $t=y$ for the walk at the left and $t=x+y$ for the walk in the middle, so that,
clearly, (\ref{CNS}) does not hold on $\left\{(x,y)\in \C^{2}:K(x,y)=0\right\}$, for any $z$.

\medskip

At last, for the walk at the right in figure~\ref{Fig_hol}, we 
have $n=4$, and we get 
     \begin{equation*}
          \sum_{k=0}^{n-1} \frac{\psi_{\delta^k}}{\prod_{i=1}^kf_{\delta^i}} =
          \frac{(y-1)(x^{2}-1)(x^{2}-y)(x^{2}-y^{2})}{xy^{4}z},
     \end{equation*}
which is obviously not identically to zero in the set $\left\{(x,y)\in \C^{2}:K(x,y)=0\right\}$, for any $z$.

\end{proof}
\appendix
\section{The basic functional equations (see \cite{FIM}, chapters 2 and 5)}\label{recall-A}
In a probabilistic framework, one considers a piecewise homogeneous random walk with sample paths in $\Zb_+^2$. In the interior of $\Zb_+^2$,  the size of the jumps is $1$. On the other hand, no assumption is made about the boundedness of the upward jumps on the
axes, neither at $(0,0)$. In addition, the downward jumps on the $x$ [resp.\ $y$] axis are bounded by $L$ [resp.\ $M$], where $L$ and $M$ are arbitrary finite integers. Then the invariant measure $\{\pi_{i,j}, i,j \ge 0\}$  does satisfy the fundamental functional equation 
\begin{equation} \label{A1-FE}
 Q(x,y) \pi(x,y) = q(x,y) \pi(x) + \widetilde{q}(x,y)
\widetilde{\pi}(y) + q_0(x,y), 
\end{equation}
with
\begin{equation*}  
\label{A2-FE}
\begin{cases}
\pi(x,y) = \DD \sum_{i,j \geq 1} \pi_{ij} x^{i-1} y^{j-1}, \\[0.5cm]  
\pi(x) = \DD \sum_{i\geq 1} \pi_{i0} x^{i-1}, \quad \widetilde{\pi}(y) = \sum_{j\geq 1} \pi_{0j} y^{j-1}, \\[0.5cm] 
Q(x,y) = \DD xy \Bigg[ 1 - \sum_{i,j\in\Sc} p_{ij} x^i y^j \Bigg], \quad  \sum_{i,j\in\Sc} p_{ij} =1, \\[0.5cm] 
q(x,y) = \DD  x \Bigg[\sum_{i \geq - 1, j \geq 0}
p'_{ij} x^i y^j - 1 \Bigg], \\[0.5cm]
\widetilde{q}(x,y) = \DD y \Bigg[\sum_{i \geq 0, j \geq -1} p''_{ij} x^i y^j - 1 \Bigg],\\[0.5cm]
 q_0(x,y) = \DD \pi_{00} \big[P_{00}(x,y)-1\big]. \end{cases}
\end{equation*}
In equation (\ref{A1-FE}), the functions $\pi(x,y), \pi(x), \widetilde{\pi}(y)$ are unknown and sought to be analytic in the region 
$\{(x,y)\in \mathbb{C}^{2} : |x|<1,|y|<1\}$ and continuous on their respective boundaries, while  the $q, \widetilde{q}, q_0$ (up to a constant) are given probability generating functions supposed to have suitable analytic continuations (as a rule, they are polynomials when the jumps are bounded). The function  $Q(x,y)$ is often referred to as the \emph{kernel} of (\ref{A1-FE}).

\section{About the uniformization of \protect $Q(x,y)=0$ (see \cite{FIM}, chapter 3)} 
\label{recall-B}
When the associated Riemann surface is of genus $1$, the algebraic curve $Q(x,y)=0$ admits a uniformization given in terms of the Weierstrass $\wp$ function with periods $\omega_1,\omega_2$ and its derivatives. Indeed, setting 
\begin{eqnarray*}
Q(x,y) & = & a(x) y^2 + b(x) y + c(x),\\
D(x) & = & b^2(x) -4a(x)c(x)=d_{4}x^{4}+d_{3}x^{3}+d_{2}x^{2}+d_{1}x+d_{0},\\
u & = & 2a(x)y+b(x),
\end{eqnarray*}
the following formulae hold (see \cite{FIM}, lemma 3.3.1).
 \begin{enumerate}
\item If $d_4 \neq 0$ ($4$ finite branch points $x_1,\ldots x_4$) then $D'(x_4)>0$ and
\begin{equation*}
\begin{cases}
x(\omega) = x_4 + \dfrac{D'(x_4)}{\wp(\omega) - \frac{1}{6} D''(x_4)},
\\[3ex] u(\omega) = \dfrac{D'(x_4) \wp'(\omega)}{2\left(\wp(\omega) -
\frac{1}{6} D''(x_4)\right)^2}. 
\end{cases}
\end{equation*}
\item If $d_4 = 0$ ($3$ finite branch points $x_1,x_2,x_3$ and $x_4=\infty$) then
\begin{equation*}
\begin{cases}
x(\omega) = \dfrac{\wp(\omega) - \dfrac{d_2}{3}}{d_3}, \\[2.5ex]
u(\omega) = - \dfrac{\wp'(\omega)}{2d_3}. 
\end{cases}
\end{equation*}
\end{enumerate}

\bibliographystyle{acm}

\begin{thebibliography}{10}

\bibitem{BK}
{\sc Bostan A. and Kauers M.}
\newblock The complete generating function for {G}essel walks is algebraic.
\newblock {\em Preprint: \verb|http://arxiv.org/abs/0909.1965|\/} (2009),
  1--14.

\bibitem{BMM}
{\sc Bousquet-M\'elou M. and Mishna M.}
\newblock Walks with small steps in the quarter plane.
\newblock {\em Preprint: \verb|http://arxiv.org/abs/0810.4387|\/} (2009),
  1--34.

\bibitem{MBM2}
{\sc Bousquet-M{\'e}lou M. and Petkovsek M.}
\newblock Walks confined in a quadrant are not always {D}-finite.
\newblock {\em Theor. Comput. Sci. 307}, 2 (2003), 257--276.

\bibitem{FIM}
{\sc Fayolle G., Iasnogorodski R. and Malyshev V.}
\newblock {\em Random walks in the quarter-plane}, vol.~40 of {\em Applications
  of Mathematics (New York)}.
\newblock Springer-Verlag, Berlin, 1999.
\newblock Algebraic methods, boundary value problems and applications.

\bibitem{FIMI}
{\sc Fayolle G., Iasnogorodski R. and Mitrani I.}
\newblock The distribution of sojourn times in a queueing network with
  overtaking: Reduction to a boundary problem.
\newblock In {\em Performance '83: Proceedings of the 9th International
  Symposium on Computer Performance Modelling, Measurement and Evaluation\/}
  (Amsterdam, The Netherlands, The Netherlands, 1983), North-Holland Publishing
  Co., pp.~477--486.

\bibitem{FLAJ}
{\sc Flajolet P. and Sedgewick R.}
\newblock {\em Analytic Combinatorics}.
\newblock Cambridge Studies in Advances Mathematics. Cambridge University
  Press, 2008.

\bibitem{FL1}
{\sc Flatto L. and Hahn S.}
\newblock Two parallel queues created by arrivals with two demands. {I}.
\newblock {\em SIAM J. Appl. Math. 44}, 5 (1984), 1041--1053.

\bibitem{Gessel}
{\sc Gessel I.}
\newblock A probabilistic method for lattice path enumeration.
\newblock {\em J. Statist. Plann. Inference 14}, 1 (1986), 49--58.

\bibitem{HUR}
{\sc Hurwitz A. and Courant R.}
\newblock {\em Funktionen Theorie}, fourth~ed.
\newblock Springer Verlag, 1964.

\bibitem{MAL}
{\sc Malyshev V.}
\newblock Positive random walks and {G}alois theory.
\newblock {\em Uspehi Matem. Nauk 1\/} (1971), 227--228.

\bibitem{Ra}
{\sc Raschel K.}
\newblock {C}ounting walks in a quadrant : a unified approach via boundary
  value problems.
\newblock {\em Preprint: \verb|http://arxiv.org/abs/|\/} (2010), 1--22.

\end{thebibliography}

\end{document}